\documentclass{article}

\usepackage{amsmath,mathtools,amsthm,amssymb}

\title{Satisfiability in Łukasiewicz logic and its unbounded relative}

\author{Zuzana Haniková\footnote{Institute of Computer Science, Czech Academy of Sciences,
Pod Vodárenskou věží 271/2
182 00 Prague 8
Czech Republic.} 
\and
Filip Jankovec\footnotemark[1] \footnote{Department of Algebra, Faculty of Mathematics and Physics, Charles University,
Sokolovská 49/83, 
186 75 Prague 8
Czech Republic.}}

\usepackage[margin=1.5in]{geometry}
\usepackage{nicefrac}
\usepackage{multicol}
\usepackage{hyperref}

\newtheorem{theorem}{Theorem}[section]
\newtheorem{lemma}[theorem]{Lemma}
\newtheorem{corollary}[theorem]{Corollary}

\newcommand{\mathrmL}{{\mathchoice{\mbox{\rm\L}}{\mbox{\rm\L}}{\mbox{\rm\scriptsize\L}}{\mbox{\rm\tiny\L}}}}

\renewcommand{\phi}{\varphi}

\newcommand{\langg}[1]{\ensuremath{\mathcal #1}}
\newcommand{\langAb}{\ensuremath{\langg{L}_{{\mathrm{Ab}}}}}
\newcommand{\langpAb}{\ensuremath{\langg{L}_{{\mathrm{pAb}}}}}
\newcommand{\langMV}{\ensuremath{\langg{L}_{{\mathrm{MV}}}}}
\newcommand{\langMVhalf}{\ensuremath{\langg{L}_{{\mathrm{MV[\frac{1}{2}]}}}}}

\newcommand{\formulas}{\ensuremath{\mathit{Tm}}}
\newcommand{\formAb}{\ensuremath{\formulas_{{\mathrm{Ab}}}}}
\newcommand{\formpAb}{\ensuremath{\formulas_{{\mathrm{pAb}}}}}
\newcommand{\formMV}{\ensuremath{\formulas_{{\mathrm{MV}}}}}
\newcommand{\formMVhalf}{\ensuremath{\formulas_{{{\mathrm{MV[\frac{1}{2}]}}}}}}

\newcommand{\alg}[1]{{\ensuremath{\boldsymbol {\mathit{#1}}}}}

\newcommand{\R}{\ensuremath{\alg{R}}}

\newcommand{\N}{\ensuremath{\mathbb{N}}}

\newcommand\standardL{\ensuremath{{[0,1]_\mathrmL}}}
\newcommand\standardLhalf{\ensuremath{[0,1]_{\mathrmL}^{{\nicefrac{1}{2}}}}}

\newcommand\depth[1]{\ensuremath{\mathrm{depth}(#1)}}

\newcommand{\tuple}[1]{\ensuremath{\langle{#1}\rangle}}  
\newcommand{\etau}{\tau'}

\newcommand{\esigma}{\sigma'}
\newcommand{\edelta}{\delta'}

\newcommand{\atoms}[1]{{\mathrm{At}^\sharp(#1)}} 

\newcommand{\boland}{\hspace{0.9pt}\mathop{\wedge\kern-8.7pt\wedge}\hspace{0.9pt}}
\newcommand{\foo}{\mathop{\bigwedge\kern-11pt\bigwedge}} 
\newcommand{\boor}{\hspace{0.9pt}\mathop{\vee\mbox{}\kern-8.7pt\vee}\hspace{0.9pt}}
\newcommand{\doo}{\mathop{\bigvee\kern-11pt\bigvee}}

\newcommand{\boneg}{\text{\raisebox{-1.1pt}{$\neg$}\mbox{}\kern-6.86pt\raisebox{0.9pt}{$\neg$}\hspace{0.9pt}}}

\begin{document}

\maketitle

\begin{abstract}
Unbounded Łukasiewicz logic is a substructural logic that combines features of infinite-valued Łukasiewicz logic with those of abelian logic. The logic is finitely strongly complete w.r.t.~the add\-itive $\ell$-group on the reals expanded with a distinguished element $-1$. We show that the existential theory of this structure is NP-complete. This provides a complexity upper bound for the set of theorems and the finite consequence relation of unbounded Łukasiewicz logic. The result is obtained by reducing the problem to the existential theory of the MV-algebra on the reals, the standard semantics of Łukasiewicz logic. This provides a new connection between both logics. The result entails a translation of the existential theory of the standard MV-algebra into itself.
\end{abstract}


\section{Introduction}
\label{s:intro}

In the realm of nonclassical logics, and particularly in its subarea known as many-valued logics, the infinite-valued {\L}ukasiewicz logic $\mathrmL$ has long occupied a prominent position. Its evolution (discussed in detail in,  e.g., \cite{Cignoli-Ottaviano-Mundici:AlgebraicFoundations, Mundici:Advanced, Hajek:1998}) spanned the introduction of its three-valued antecedent by {\L}ukasiewicz in his famous analysis of future contingents \cite{Lukasiewicz:onDeterminism}, as well as the infinite-valued calculus  also introduced by {\L}ukasiewicz and elaborated by himself and Tarski in the paper \cite{Lukasiewicz-Tarski:Untersuchungen}, which first discussed, but did not prove, completeness w.r.t.~the intended semantics.
A completeness theorem was given much later by Rose and Rosser \cite{Rose-Roser:Fragments}.  MV-algebras, the equivalent algebraic semantics, were introduced and analyzed by Chang, and completeness was proved \cite{Chang:MVAlgebras,Chang:ANewProof}.  
Numerous monographs deal exclusively or substantially in the topic, 
e.g., \cite{Cignoli-Ottaviano-Mundici:AlgebraicFoundations,Hajek:1998,Metcalfe-Olivetti-Gabbay:ProofTheoryFuzzy,Mundici:Advanced,HandbookMFL-IandII}.   
Moreover, {\L}ukasiewicz logic and its extensions can fruitfully be investigated in the framework of substructural logics \cite{Restall:Substructural,Galatos-JKO:ResiduatedLattices}.

Increasingly then, the well-established area of  {\L}ukasiewicz logic and MV-algebras has been taken  as a vantage point  for various theoretical or applied problems, 
be it expansions of its language  with constants in the pioneering works of Pavelka \cite{Pavelka:Fuzzy},  modal expansions
  \cite{Fitting:MV-ModalLogics,Fitting:MV-ModalLogicsII,Hajek-Harmancova:MVModalLogic,Vidal:UndecidabilityNonAxModalMVL}, or
reasoning about probabilities \cite{Flaminio-Montagna:ConditionalProbability}, to name a few.

Mundici initiated the study of computational complexity problems raised by {\L}ukasiewicz logic in his famous work \cite{Mundici:Satisfiability} which showed  satisfiability and positive satisfiability of terms in the MV-algebra on the real unit interval with the natural order, $\standardL$,  to be NP-complete. An earlier important work, relevant to many subsequent contributions to
the area of algorithmic problems in {\L}ukasiewicz logic, came in McNaughton's work \cite{McNaughton:FunctionalRep} that characterized term-definable functions in $\standardL$ as continuous, piecewise linear functions with integer coefficients.  Algebraic methods are the mainstay of studying computational problems in the area, see especially \cite{Aguzzoli:PhD, Aguzzoli-Ciabattoni:Finiteness, Aguzzoli-Gerla:finitereductions, Aguzzoli:asymptoticBoundLuk}.
Komori's classification of subvarieties of MV \cite{Komori:SuperLukasiewiczImplicational,Komori:SuperLukasiewiczPropositional} yielded the coNP-completeness result for axiomatic extensions of {\L} \cite{Cintula-Hajek:ComplexityLukasiewicz}, 
and Mundici's NP-completeness result for term satisfiability generalizes to the existential theory of $\standardL$  \cite{Hanikova:Handbook}.
While Mundici's results place the satisfiability and the tautology problems 
in $\mathrmL$  on a par with their counterparts in classical propositional logic (either being NP-complete and coNP-complete respectively), this is not a given in the infinite-valued setting.
More generally, {\L}ukasiewicz logic is  dissimilar from classical logic in many respects and  a list to showcase this phenomenon ought not to omit $\mathrmL$ not being structurally complete \cite{Dzik:UnificationBLHoops}, the admissible rules being PSPACE complete \cite{Jerabek:ComplexityAdmissibleRulesLuk}, the first-order standard tautologies
not being recursively enumerable \cite{Scarpellini:PredikLukNonComplete} and in fact $\Pi_2$-complete in the arithmetical hierarchy \cite{Ragaz:PhD}, or the complex structure of  the lattice of extensions of {\L}, which is dually isomorphic to the lattice of subquasivarieties of the variety of MV-algebras. 
This lattice is Q-universal \cite{AdamsDziobiak:Q-universal}. Consequently, the extensions cannot all be decidable for cardinality reasons; 
this is also the case already for term satisfiability problems for various MV-algebras on the rational domain \cite{Hanikova-Savicky:SAT}.  

Unbounded {\L}ukasiewicz logic $\textrm{{\L}u}$ was considered recently in the papers \cite{Cintula-GNS:DegreesToEleven,Cintula-Jankovec-Noguera:SuperabelianLogics}. Formally, it is obtained by expanding the language of lattice-ordered abelian groups with a propositional constant $-1$ with suitable axioms ensuring that in any algebra within its equivalent algebraic semantics, the upset of the interpretation of $0$ is the set of designated elements, representing truth, whereas the downset of $-1$ represents falsity.
By design, then, the logic bears resemblance to both {\L}ukasiewicz logic and to abelian logic,
as considered in Meyer and Slaney's paper  \cite{Meyer-Slaney:AbelianLogic}.   Unlike {\L}, and like abelian logic,  
\textrm{{\L}u} is inconsistent with classical logic.
 
Weispfenning \cite{Weispfenning:ComplexityAbelianLGroups} proved that the existential theory of the variety  of $\ell$-groups is NP-complete. So the universal theory of  $\ell$-groups is coNP-complete and \emph{a fortiori}, the quasiequational theory and the equational theory are in coNP. It is also known that $\ell$-groups are generated as a quasivariety by the totally ordered members \cite{Clifford:PartiallyOrderedAbelianGroups}, in fact, by any nontrivial totally ordered member, since the universal theories of each two nontrivial totally ordered abelian groups coincide \cite{Gurevich-Kokorin:UniversalEquivalence}. One easily shows that the equational theory of the totally ordered additive group on the integers is coNP-hard, thus any of the above universal fragments is coNP-complete.

The purpose of this paper is to provide a complexity estimate for 
\textrm{{\L}u} by exhibiting a polynomial-time reduction of the existential theory of the intended semantics of \textrm{{\L}u}---the algebra $\R_{-1}$, the pointed abelian totally ordered additive group on the reals where ${-1}$ is interpreted by itself---to the existential theory of the intended semantics of \textrm{{\L}}, namely $\standardL$, the standard MV-algebra on the reals. 
In both cases, the logics are finitely strongly complete w.r.t.~the intended semantics. 
In our opinion the method surpasses the result in terms of interest, since the result might be anticipated.  
The method consists in taking a linear map of a sufficiently large neighbourhood of the element $0$ to the interval $[0,1]$, with the coefficient of the linear map  determined by the instance (i.e., a given existential sentence).  

In combination with the extant method of \cite{Cintula-Jankovec-Noguera:SuperabelianLogics} that reduces in polynomial time the set of valid universal formulas (valid identities) of $\standardL$ to the set of valid universal formulas (valid identities respectively) of the algebra $\R_{-1}$ one obtains firstly, a coNP-completeness of each of the universal fragments (i.a., the logic  \textrm{{\L}u} is coNP-complete) and secondly, a non-trivial interpretation of the existential theory of the 
standard MV-algebra in itself.

The result invites a comparison with the work of Marchioni and Montagna \cite{Marchioni-Montagna:ComplexityDefinabilityLPi}, which concerns the logic $\text{\L}\Pi\nicefrac{1}{2}$ and its relation to the universal theory of real closed fields (RCF). 
The authors show that the universal fragment of RCF has a polynomial time reduction
to the (theorems of) the logic $\text{\L}\Pi\nicefrac{1}{2}$.
The relevance consists in, on the one hand, the logic $\text{\L}\Pi\nicefrac{1}{2}$
being a conservative expansion of the logic {\L} (the theorems of {\L} are exactly those theorems of $\text{\L}\Pi\nicefrac{1}{2}$ that only use the language $\langMV$) and on the other hand, the universal theory of pointed abelian totally ordered group being term-equivalent to the multiplication- and division-free fragment of the universal theory  of  RCF. 
Their reduction starts with fixing a bijection between $\R$ and the interval $(0,1)$ and then defining the operations of the real closed field on $\R$ in the  $\text{\L}\Pi\nicefrac{1}{2}$-algebra on the reals. This approach cannot be used in the present case because
it relies on the multiplicative function symbols of the logic $\Pi$ even when defining  the real addition; so it does not scale well to language fragments.

This paper is structured as follows. Key notions are presented in section \ref{s:prelim}. Section \ref{s:bounded} establishes that assignments that make an open formula of the language of pointed $\ell$-groups true, if any, can be found in a suitably bounded interval. Section \ref{s:reduction} presents the main result, the polynomial-time reduction. Section 
\ref{s:embeddings} then composes two reductions in order to obtain a nontrivial reduction
of the existential theory of $\standardL$ into itself. 

\section{Preliminaries}
\label{s:prelim}

Let $\mathrm{Var}=\{x_1, x_2, \dots\}$ be a denumerable set of variables. 
A first-order language $\langg{L}$ uniquely determines the set $\formulas_{\langg{L}}$ of terms of $\langg{L}$. The logical symbol $=$ is the only predicate symbol of any algebraic language; however, we may additionally use binary predicate symbols $\leq$, $<$ (in Section \ref{s:bounded}). If $\phi$ and $\psi$ are terms of $\langg{L}$, then $\phi \diamond \psi$ is 
an atomic formula (or ``atom'') of $\langg{L}$ for $\diamond\in\{=,\leq,<\}$.
 If there are no predicate symbols in $\langg{L}$ then $\langg{L}$-atoms are just $\langg{L}$-identities $\phi=\psi$.
 We use lowercase Greek for terms and uppercase Greek for first-order formulas.
For boolean symbols occurring in first-order formulas (boolean combinations
of atoms of $\langg{L}$) we use
 $\boneg$ for boolean negation, 
 $\boland$ and $\boor $ for boolean conjunction and disjunction respectively, 
 $\foo$ and $\doo$ for finite boolean conjunctions and disjunctions respectively.
An open formula of $\langg{L}$ is an arbitrary finite boolean combination (i.e., 
a $ \{ \boneg,\boland,\boor \} $-combination)  of $\langg{L}$-atoms; we do not in general assume a normal form for first-order boolean formulas.
Moreover, if $\Rightarrow$ denotes  boolean implication, 
quasiidentities of $\langg{L}$ are formulas of the form $\foo_{i=1}^n (\phi_i=\psi_i) \Rightarrow (\phi_0=\psi_0)$ where $\{\phi_i,\psi_i\}_{0\leq i\leq n}$ is a set of $\langg{L}$-terms. 
Since quasiequational theories are not of particular interest in this work, $\Rightarrow$ is not taken as primitive and a formula $\Phi\Rightarrow\Psi$
is understood as $\boneg\Phi\boor\Psi$  with only an insubstantial increase in length.
An existential (universal) sentence of a language $\langg{L}$ is an existential (universal respectively) 
closure of an open formula of $\langg{L}$.  

If $\phi$ is an $\langg{L}$-term, $\mathrm{Var}(\phi)$ denotes the set of variables that occur in $\phi$, and analogously for $\mathrm{Var}(\Phi)$ if $\Phi$ is an  $\langg{L}$-formula. 
We use $\bar x$ for  $\tuple{x_1,\dots,x_n}$ if $n$ is of no consequence.

The following languages will be used:\footnote{We write $\formAb$ instead of $\formulas_{\langAb}$. We use the same notation for a function symbol of a language and for its interpretation in a particular algebra where possible, to avoid cluttering.}  
\begin{itemize}
\item $\langAb = \tuple{+,-,0,\land,\lor}$ is the language of abelian lattice-ordered groups ($\ell$-groups) and $\formAb$ is the corresponding set of well-formed terms;
\item $\langpAb =\tuple{+,-,0,-1,\land,\lor}$ is the language of pointed $\ell$-groups and
$\formpAb$ the set of its well-formed terms;
\item $\langMV = \tuple{\oplus,\otimes,\neg,\rightarrow,0,1,\land,\lor}$ is the language of MV-algebras and $\formMV$ the set of its well-formed terms;
\item $\langMVhalf = \tuple{\oplus,\otimes,\neg,\rightarrow,0,1,\nicefrac{1}{2},\land,\lor}$ expands the language of MV-algebras with a new constant $1/2$ and $\formMVhalf$ is the set of its well-formed terms.
\end{itemize}

$\R=\tuple{\mathbb R,+,-,\land,\lor,0}$ is an abelian lattice-ordered group ($\ell$-group) on the real numbers with the usual order, (the usual) addition, subtraction and $0$. The algebra $\R$ provides a semantics to the language $\langAb$. 

$\R_{-1}=\tuple{\mathbb R,+,-,\land,\lor,0,-1}$ is a \textit{pointed} abelian $\ell$-group whose $\langAb$-reduct is $\R$ and the element $-1$ interprets the constant $-1$.\footnote{While the algebra $\R_{-1}$ seems very similar to $\R$, the additional constant in the language grants much more expressivity and brings the logic $\textrm{{\L}u}$ closer to the logic {\L}.}

In the \emph{standard MV-algebra} $\standardL$ on the reals, which provides an interpretation to the language $\langMV$, the domain is $[0,1]$, as for the basic operations, $\neg a$ is $1-a$ and $a\otimes b$ is $\max(0,a+b-1)$ for $a,b\in[0,1]$. The remaining function symbols are definable. 
\setlength{\multicolsep}{5pt} 
\begin{multicols}{2}
\begin{itemize}
    \item $a\oplus b$ is $\neg(\neg a \otimes \neg b)$
    \item $a\to b$ is $\neg a \oplus b$
    \item $a\land b$ is $a\otimes (a\to b)$
    \item $a\lor b$ is $(a\to b)\to b$
    \item $0$ is $a\otimes \neg a$ for some $a$ 
    \item $1$ is $\neg 0$
   \end{itemize}
\end{multicols}
In any MV-algebra, $\neg\neg a = a$ holds (involutive law), analogously in any abelian $\ell$-group, we have $-\!-a=a$. In view of this, we assume that terms are given in a reduced form within their respective languages, containing no subterms of the form $\neg \neg\phi$ or $-\!-\phi$, respectively.

The two element boolean algebra is a subalgebra of $\standardL$.
The algebra $\standardLhalf$ interprets the language $\langMVhalf$. Its $\langMV$-reduct is the algebra $\standardL$ and the element $\nicefrac{1}{2}$ interprets the constant $\nicefrac{1}{2}$.

It is assumed throughout that no well-formed term of the languages under consideration contains two consecutive occurrences of the symbol $-$ or the symbol $\neg$. This can be assumed without a loss of generality, relying on the valid involutive laws.  

One can introduce the (infinite-valued) {\L}ukasiewicz logic $\mathrmL$ 
and the unbounded {\L}u\-ka\-sie\-wicz logic $\mathrmL\mathrm{u}$ axiomatically and then obtain 
that $\mathrmL$ is finitely strongly complete w.r.t.~$\standardL$ and 
$\mathrmL\mathrm{u}$ is finitely strongly complete w.r.t.~$\R_{-1}$.
We take these results for granted, referring the reader to the works \cite{Cignoli-Ottaviano-Mundici:AlgebraicFoundations,  Hajek:1998,Cintula-Jankovec-Noguera:SuperabelianLogics}. Moreover, without further ado we use Mundici's result on the NP-completeness of the term satisfiability problem in $\standardL$ and the coNP-completeness of term validity therein.
Recall that the designated element of $\standardL$ is $1$; accordingly, a term $\phi(\bar x)$ holds under  an assignment $\bar a$ provided that $\phi(\bar a)= 1$,
and $\phi(\bar x)$ is \emph{satisfiable} in $\standardL$ provided that an assignment exists that makes it true there. 
We retain this terminology also for first-order open formulas: such a formula $\Phi(\bar x)$ is satisfiable in $\standardL$ provided  an assignment to its variables exists that makes it true therein.
A term $\phi$ of $\langMV$ is \emph{valid} in $\standardL$ provided it is true under any assignment therein.
The set of designated elements of $\R_{-1}$ is the set of its nonnegative elements (the upset of $0$), 
and accordingly a term is satisfiable in $\R_{-1}$ provided some assignment sends it to a nonnegative element; the rest is analogous.

If $\langg{L}$ is a language and $\alg{A}$ is an $\langg{L}$-algebra, 
the existential theory ($\exists$-theory) of $\alg{A}$ is the set of all existential $\langg{L}$-sentences valid in $\alg{A}$; dually for the universal theory ($\forall$-theory) and universal $\langg{L}$-sentences. 
The (quasi)equational theory of $\alg{A}$ is the set of universal closures of $\langg{L}$-(quasi)identities
valid in $\alg{A}$.
Both the equational theory and the quasiequational theory of an algebra $\alg{A}$ are fragments of the universal theory of $\alg{A}$.

Each term $\phi$ of $\langAb$ in $n$ variables $x_1,\dots,x_n$ in $\mathrm{Var}$ 
uniquely determines a function $f_\phi\colon \mathbb R^n\to \mathbb R$ (on the domain of the algebra $\R$) 
using the following inductive steps:
\begin{multicols}{2}
\begin{itemize}
    \item $f_x=x$ for each variable $x$ in $\mathrm{Var}$;
    \item $f_{\phi+\psi}=f_{\phi}+f_{\psi}$;   
    \item $f_{-\phi}=-f_{\phi}$;  
    \item $f_{\phi \land \psi}=\min(f_{\phi},f_{\psi})$;
    \item $f_{\phi \lor \psi}=\max(f_{\phi},f_{\psi})$;
    \item $f_0=0$.
\end{itemize}
\end{multicols}
Analogously for a term $\phi$ of $\langpAb$.
Furthermore, a term 
$\phi$ of $\langMV$ uniquely determines a function  $f_\phi\colon [0,1]^n\to [0,1]$ 
(on the domain of the standard MV-algebra), i.e., 
\begin{multicols}{2}
\begin{itemize}
   \item $f_{\phi \otimes \psi}=\max(0,f_{\phi}+f_{\psi}-1)$,
   \item $f_{\neg \phi}=1-f_{\phi}$,
\end{itemize}  
\end{multicols}
etc., and analogously for a term of the language $\langMVhalf$.

For a term $\phi$ of $\langAb$, the \emph{depth} of $\phi$ is defined by induction:
$\depth{0}=0$; $\depth{x}=0$ if $x\in\mathrm{Var}$; 
$\depth{-\phi}=\depth{\phi}+1$; and
$\depth{\phi\circ\psi}= \max( \depth{\phi},\depth{\psi} )+1$ if $\circ$ is one of $\{+,\land,\lor \}$. Analogously for the other languages, with depth of constants set to $0$.

For a term $\phi$  and a first-order formula $\Phi$ of $\langAb$,  
$\sharp(\phi)$ and $\sharp(\Phi)$ will denote the \emph{length} (or  \emph{size}) of $\phi$ and $\Phi$ respectively. 
We define $\sharp(\phi)$ as the total number of occurrences of variables and constants in $\phi$; should one view $\phi$  as a tree, $\sharp(\phi)$ is the number of leaves. Recalling that in $\phi$ no two consecutive negations `$-$' can occur, the number of occurrences of function symbols in $\phi$ (i.e., of the inner nodes of the tree) is bounded by $3\sharp(\phi)$.
The sizes of representations of each variable and each constant are neglected and representation  in unit length is considered.

For an open $\Phi$, we use $\atoms{\Phi}$ for  the number of occurrences of atomic formulas. 
We  define  $\sharp(\Phi)$ to be the sum of the sizes of all the occurrences of terms in the atomic identities; 
we have  $\sharp(\Phi) \leq \atoms{\Phi} \cdot 2 \max\{ \sharp(\phi) \mid \phi \text{ atom in } \Phi\}$.
Again any two consecutive boolean negations are omitted.  Analogous considerations apply to the remaining languages.

Let $\langg{L}$ be a language and $\Phi$ an open $\langg{L}$-formula.
A \emph{Tseitin variant} of $\Phi$ is an open formula $\Phi'$ with some desirable properties:
 $\sharp(\Phi')$ is of polynomial size in  $\sharp(\Phi)$; 
 every atom of  $\Phi'$ contains at most one function symbol (in particular, terms of $\Phi'$ have depth at most 1); 
 and $\Phi'$ is equisatisfiable with $\Phi$ in any $\langg{L}$-structure $\alg{M}$ (i.e., the existential closure of $\Phi$ holds in $\alg{M}$ iff so does the existential closure of $\Phi'$). We proceed to define Tseitin variants formally. 

Let $T$ be the set of all subterms of $\langg{L}$-terms that occur in $\Phi(x_1,\dots,x_n)$.
Let $|T|$ be the cardinality of $T$.
To each term $\phi$ in $T$ that is not a variable among $x_1,\dots, x_n$,
 assign a new variable $x_\phi$ from $\mathrm{Var}$, not occurring in $\Phi$, with  distinct variables  assigned to distinct terms.  
 We have $|T| \leq 4\sharp(\Phi)$.
Define a new open formula $\Psi$ from $\Phi$ as follows. For every atom $\phi \diamond \psi$ occurring in $\Phi$ (here $\diamond$ is among $\{=,<\}$),  if either of $\phi$ or $\psi$ is not a variable among $x_1,\dots, x_n$, replace this occurrence of $\phi$ or $\psi$ with the new variable $x_\phi$ or $x_\psi$ respectively. For convenience, if the term $\phi$ is a variable, $x_\phi$ denotes the original variable $\phi$.
Then the Tseitin variant of $\Phi$ is

$$ \Phi' \coloneqq \Psi \boland\foo_{\substack{\phi \text{ occurs in } \Phi  \\ \phi\text{ is } f(\phi_1,\dots,\phi_n) \text{ for some $n$-ary function symbol } f \in \langg{L} }} x_\phi=f(x_{\phi_1},\dots, x_{\phi_n}),$$ 
where $x_\phi, x_{\phi_1},\dots, x_{\phi_n}$ are metavariables for elements of $\mathrm{Var}$.
Observe $\Phi'$ has no compound terms. 
Furthermore, $\atoms{\Phi'} \leq\atoms{\Phi}+|T|$. 
Lastly, $\Phi'$ is indeed  equisatisfiable with $\Phi$ in any $\langg{L}$-structure $\alg{M}$: 
for any assignment $v_\alg{M}$ in $\alg{M}$ one has that $v_\alg{M}$ satisfies 
$\Phi$ in $\alg{M}$ if and only if $\Phi'$ is satisfied by the assignment $v_\alg{M}'$, where $v_\alg{M}'$ extends  $v_\alg{M}$ by sending each added variable $x_\phi$ to $v_\alg{M} (\phi)$.

\section{Bounded models}
\label{s:bounded}

This section establishes a technical result to be used below: 
if an existential sentence $\exists x_1 \exists x_2 \ldots \exists x_n \Phi$ 
in the language $\langpAb$ holds in $\R_{-1}$, then there is a satisfying (rational) assignment
whose binary representation has size polynomial in $\sharp(\Phi)$.
Our exposition follows (in some detail) the one in the paper \cite{Fagin-Halpern-Megiddo:LogicReasoningProb} as to the general method. 
  
\begin{lemma} \label{l:small solution linear programming}
    Let $k,m,n$ be positive natural numbers.
    Let $Ax \leq b$ be a system of linear inequalities with $A$  an $n \times m$ integer matrix, $b$ a rational $n$-tuple, and $a_{i,j},b_i \in [-k,k]$ for $i\leq n,j \leq m$. 
    Assume the system has a solution in $\mathbb R$. 
    Then there is a (rational) solution in the interval $[-(mk)^m,(mk)^m]$.
\end{lemma}

\begin{proof}
    The  polytope  $P=\{x \mid Ax\leq b\}$ is nonempty. 
    Fix a non-empty face of $P$ minimal under inclusion. 
    Such a minimal face is given by a system of equations $\hat{A} x=\hat{b}$, a subsystem of $Ax=b$. To solve this system, it is enough to solve a system  $A^0 y=b^0$, where $A^0$ and $b^0$ are obtained from $\hat{A}$ and $\hat{b}$ by fixing $d$ columns that contain pivots after Gauss elimination
    (permuting columns if necessary, assume the first $d$ columns of $A$) and fixing the corresponding $d$ linearly independent rows. The matrix $A^0$ is regular of rank $d$ and  $b^0$ is a $d$-tuple of rationals. A solution of $A^0 y=b^0$ also yields a solution of $\hat{A} x=\hat{b}$ by assigning the value $0$ to all
    variables $x_{d+1},\dots, x_m$. 
     Cramer's rule yields the solution 
     $$\tuple{y_i}_{i \leq d}=\tuple{\frac{\det(A^0_{i})}{\det(A^0)}}_{i\leq d}.$$ 
     Since $A$ has integer entries, $|\det(A^0)| \geq 1$ and thus  $|y_i| \leq |\det(A^0_i)|$ for each $i \leq d$.
    We have $|\det(A^0_i)| \leq d! \cdot k^d \leq d^d \cdot k^d=(kd)^d$.
    So $y_i \in [-(dk)^d,(dk)^d]$ for each $i \leq d$, hence for each $i \leq m$.
    Since $d \leq \min(m,n)$ we have $x_i \in [-(mk)^m,(mk)^m]$ for $i \leq m$.
\end{proof}

\begin{lemma}[Bounded assignment] \label{t:small models}
    There exists an integer $C$ such that for every open formula $\Phi(x_1,\dots,x_n)$ 
    of the language $\langpAb$ with $N=\sharp(\Phi)$, the following holds:

    If $\R_{-1}\models\exists x_1\dots\exists x_n\Phi$, then there are 
    real numbers $\tuple{a_1,\dots,a_n}$ and an integer $M$ such that 
    $\Phi(a_1,\dots,a_n)$ holds in $\R_{-1}$, 
    $-2^M \leq a_i\leq 2^M$ for each $1\leq i\leq n$,
    and $M \leq C N \log_2 N$.


\end{lemma}

\begin{proof} 
For convenience, the symbol and $<$ is added to the language $\langpAb$. 
Let $\Phi(x_1,\dots,x_n)$ be given. 
Applying de Morgan rules and removing any double boolean negations that occur, 
$\Phi$ can be rewritten so that any boolean negation is applied to an atom. Using trichotomy, any atom $\boneg(\phi=\psi)$ can be replaced with  $(\phi<\psi) \boor\, (\psi<\phi)$. 
This yields an equivalent negation-free formula $\Phi'(x_1,\dots,x_n)$ whose atoms are identities and strict inequalities. Moreover $\atoms{\Phi'}\leq 2\atoms{\Phi}$.

 $\Phi'$ can be  brought to a DNF without negations, obtaining a 
$$\Phi''=\Psi_1\boor\Psi_2\boor\dots\boor\Psi_m$$ for some $m$, where each $\Psi_j$ for $1\leq j\leq m$ is a conjunction of atoms of the form $(\phi=\psi)$ or $(\phi<\psi)$. 
The value of $m$ need not be polynomial in $\atoms{\Phi'}$.
The remaining part of the proof establishes that a bounded satisfying assignment to $\Phi$ exists 
on the basis of the existence of the formula $\Phi''=\mathrm{DNF}(\Phi')$ and the fact, which we proceed to show, 
of an existence of a bounded satisfying assignment to at least one among the $\Psi_j$.
Both $\Phi'$ and $\Phi''$ are equivalent to $\Phi$; namely, a tuple
$\tuple{a_1,\dots,a_n}$ makes $\Phi$ true in $\R_{-1}$ if and only if it makes $\Phi'$ true, and the latter is the case if and only if it makes some $\Psi_j$ true for some $j\leq m$. 
Trivially $\atoms{\Psi_j}\leq\atoms{\Phi'}$ for $1\leq j\leq m$.

Next, we remove any complex terms from $\Phi''$ using a Tseitin transformation. 
Let $\Psi'_1,\dots \Psi'_m$ be Tseitin variants of $\Psi_1,\dots \Psi_m$. Without loss of generality, we assume that the Tseitin transformation assigns the same variable $x_t$ to identical terms $t$, regardless of the disjunct in which they appear. Consequently, $|\bigcup_{j=1}^m {\mathrm Var}(\Psi'_j)|= |T|$, where $|T|$ is the cardinality of the set $T$ of all $\langpAb$-(sub)terms that occur in $\Phi'$.

Let us set
$$(\Phi''')=(\Psi_1' \boor \dots \boor \Psi_m').$$
Since $\Psi_j$ is equisatisfiable with $\Psi'_j$ for each $j\leq m$, we obtain that $\Phi''$ is equisatisfiable with $\Phi'''$.  
Let $k$ denote the total number of  variables in $\Phi'''$. Clearly, we have  $k = |T| \leq  4 \sharp(\Phi')$.

The formula $\Phi'''$ is still in DNF and we have $\atoms{\Psi_j'} \leq |T|+ \sharp(\Psi_j)\leq 5\sharp(\Phi')$ for each $j\leq m$.

By the assumption, there is an assignment $v$ in $\R_{-1}$ with $v(x_i)=a_i$ for $i\leq k$ such that the vector $\tuple{a_1,\dots, a_k}$  makes both $\Phi$ and $\Phi'''$ true in $\R_{-1}$.
Fix such an assignment $v$.
Then there is a total ordering given by a permutation $\sigma$ of $\{1,\dots,k\}$ such that, in $\R$, $a_{\sigma(i)} \leq a_{\sigma(i+1)}$ for each $i<k$.
Let
$$\rho (x_1, \dots, x_k) \coloneqq \foo_{{\{i<k\}; \; a_{\sigma(i)}<a_{\sigma(i+1)}}} (x_{\sigma(i)} < x_{\sigma(i+1)}) \boland \foo_{\{i<k\} ; a_{\sigma(i)}=a_{\sigma(i+1)}} (x_{\sigma(i)}=x_{\sigma(i+1)}).$$
We have  $\atoms{\rho} \leq k$.
For some $j\leq m $, $(\rho \boland \Psi_j')(a_1,\dots,a_k)$ is true in $\R_{-1}$.
 
Fix  such  $j\leq m$
and let us show that there exists an assignment that makes the formula $(\rho \boland \Psi_j')$ true $\R_{-1}$ and whose values are bounded as required. 
We refer to the formula $\Psi_j'$ simply as $\Psi$.

Using $\rho$, one can omit atoms with the function symbols $\land$ and $\lor$.
Indeed, since  $\rho\boland\Psi$ is true under $v$, 
it must also be the case that 
if, e.g., $x_\phi \land x_\psi = x_\chi$ occurs as an atom in $\Psi$
(which means either that the term $\chi$ is the term $\phi\land \psi$ in $\Phi$,
or that $\chi=\phi\land\psi$ is an identity in $\Phi$), 
then the identity $\phi(a)  \land \psi(a) = \chi(a)$ holds and hence, 
this is captured by $\rho$, namely $x_\chi =_\rho \min_\rho( x_\phi,x_\psi)$.
The other cases featuring $\lor$ are analogous. 
Thus $\rho$ already captures all the information
that is provided by those atoms in  $\rho \boland \Psi$ in
which the lattice function symbols $\land$ and $\lor$ 
occur, whereby these atomic formulas are redundant and can be omitted from the conjunction.

Such an omission yields a new conjunction of identities and strict inequalities 
$\rho \boland\Psi^\star$.
This formula has $k$ variables, and 
$$l\coloneqq \atoms{\rho \boland\Psi^\star} = 
\atoms{\rho} + \atoms{\Psi_j'}  \leq 
 \sharp(\Phi')   
+ 5\sharp(\Phi') 
=6\sharp(\Phi')\leq 12\sharp(\Phi).$$ 
Now $\rho \boland\Psi^\star$ defines a system $S$ of linear equations and strict 
inequalities in $\R$ as follows.
Any atomic formula in $\rho \boland\Psi^\star$ has at most three distinct variables 
and at most one function symbol. Subtracting the expression on right-hand side of each atomic
formula from both sides, we get a system consisting of a subsystem $Ax = c$ with $l_1$  rows,
and a subsystem $Bx <d$ with $l_2$  rows, both in the variables $x_1,\dots, x_k$ and with $l_1+l_2=l$, thus $l_1, l_2\leq 12\sharp(\Phi)$.
Since $\tuple{a_1,\dots,a_k}$ is a solution to the system,
there has to exist a rational vector $d'=\tuple {d_1',\dots,d_{l_2'}}$ where
$d_i-1\leq d_i'< d_i$, such that a modified system $S'$ consisting of $Ax = c$ and $Bx \leq d'$ still has a solution.
Clearly, any solution to $S'$ is also a solution to $S$.

Rewriting each of the equations in  $Ax = c$ as two inequalities and adding the rows
of $Bx \leq d'$ yields a system $ A' x \leq b' $,  
s.t.~there is an integer $c\leq 36$ with 
the number of both rows and columns in  $ A' x \leq b' $  bounded
by $C \cdot \sharp(\Phi)$.
$A'$ is an integer matrix, $b'$ is a rational vector, 
and the absolute values of entries in either are bounded by  $3$.
Applying Lemma~\ref{l:small solution linear programming} yields
a solution whose absolute value is bounded by  $(3C \sharp(\Phi))^{C\sharp(\Phi)}$.
It is then enough to pick $M = \lceil  C\sharp(\Phi) \log_2 (3C\sharp(\Phi)) \rceil$.
\end{proof}

Using the tools of this section, one might proceed to show that the existential theory of $\R_{-1}$ is in NP.
Indeed, one gets almost for free from the above, for any open  $\langpAb$-formula $\Phi$ with a real assignment that satisfies it, 
also a satisfying assignment that is rational with  both the numerators and the denominators  of polynomial size in $\sharp(\Phi)$. This assignment is then a witness of the fact that the existential closure of $\Phi$ is true, whence the existential theory of $\R_{-1}$ is in NP.
However, our intent is to use the results of this section to obtain suitable maps on the reals in the next section, and hence 
a  polynomial-time reduction of the existential theory of $\R_{-1}$ to the existential theory of $\standardL$.
In other words, we answer the satisfiability question, the instances of which are open formulas, indirectly.
It is worth remarking that, if one established NP-completeness of the existential theory of $\R_{-1}$ anyhow, 
then one would still get a poly-time reduction between this problem and the existential theory of $\standardL$,
because the latter problem is known to be NP-complete (using, e.g., \cite{Hahnle:MIP,Hanikova:Handbook}). Our reduction has the added value of being explicit.

\section{ \texorpdfstring{Reducing the $\exists$-theory of $\R_{-1}$ to the $\exists$-theory of $\standardL$}{Reducing the ∃-theory of R1 to the ∃-theory of L}
}
\label{s:reduction}

Consider the mapping $\tau \colon \formAb \rightarrow \formMVhalf$, assigning to each term $\phi$ of $\langAb$ a term of $\langMVhalf$ as follows:
\setlength{\columnsep}{-40pt}
\begin{multicols}{2}
\begin{enumerate}
    \item $\tau(x)=x$ for $x \in {\mathrm Var}(\phi)$;
    \item $\tau(0)=\frac{1}{2}$; 
    \item $\tau(-\phi)=\neg(\tau(\phi))$;
    \item $\tau(\phi+\psi)=(\tau(\phi) \oplus \tau(\psi)) \otimes (\frac{1}{2} \oplus (\tau(\phi) \otimes \tau(\psi)))$;
    \item $\tau(\phi \lor \psi)=\tau(\phi) \lor \tau(\psi)$;
    \item $\tau(\phi \land \psi)=\tau(\phi) \land \tau(\psi)$. 
\end{enumerate}
\end{multicols}

For each pair of nonnegative integers $M$ and $k$, let $r_{M,k}:\mathbb R \rightarrow \mathbb R $ be defined as
\begin{equation}\label{eq:2}
r_{M,k}(a)=\frac{a}{2^{M+k+1}}+\frac{1}{2}.
\end{equation}

\begin{lemma}
\label{lm_r_m_k} \label{l:tau}
Let $M$ and $k$ be fixed nonnegative integers. Let $\phi(x_1,\dots, x_n)$ be a term in 
$\formAb$ with $n$ variables, let $\depth{\phi}=l$ with $l \leq k$. Assume  $a_1,\dots, a_n \in [-2^{M},2^{M}].$ Then
\[
f_{\tau(\phi)}(r_{M,k}(a_1),\dots,r_{M,k}(a_n))=r_{M,k}(f_{\phi}(a_1,\dots,a_n)).
\]
Moreover, $r_{M,k}(f_{\phi}(a_1,\dots,a_n)) \in [\frac{1}{2}-\frac{1}{2^{k-l+1}},\frac{1}{2}+\frac{1}{2^{k-l+1}}]$.
\end{lemma}

\begin{proof}
    First let us note that for each $a \in [-2^M,2^M]$ it holds that $$r_{M,k}( a) \in [\frac{1}{2}-\frac{1}{2^{k+1}},\frac{1}{2}+\frac{1}{2^{k+1}}].$$  We will proceed by induction on term depth.
    \begin{itemize}
    \item Let $\phi$ be a variable $x$. By definition we have 
    $$f_{\tau(\phi)}(r_{M,k}(a_1),\dots,r_{M,k}(a_n))=r_{M,k}(f_{\phi}(a_1,\dots,a_n)) \in [\frac{1}{2}-\frac{1}{2^{k+1}},\frac{1}{2}+\frac{1}{2^{k+1}}].$$
    \item Let $\phi=0$. We have  $$f_{\tau(\phi)}(r_{M,k}(a_1),\dots,r_{M,k}(a_n))=\frac{1}{2}=r_{M,k}(f_{\phi}(a_1,\dots,a_n)).$$
    \end{itemize}
Now assume $\depth{\phi}=l$ and the statement holds for any $l'<l$.
    \begin{itemize}
\item   Let $\phi$ be $-\psi$. We assume $f_{\tau(\psi)}(r_{M,k}(a_1),
        \dots,r_{M,k}(a_n))=r_{M,k} (f_{\psi}(a_1,\dots,a_n))$. \linebreak
        We have 
        \begin{align*}
            f_{\tau(-\psi)}(r_{M,k}(a_1),\dots,r_{M,k}(a_n))&= f_{\neg \tau(\psi)}(r_{M,k}(a_1),\dots,r_{M,k}(a_n))\\ = 1- f_{\tau(\psi)}(r_{M,k}(a_1),\dots,r_{M,k}(a_n))&=1- r_{M,k}(f_{\psi}(a_1,\dots,a_n))\\ =
        1-(\frac{f_{\psi}(a_1,\dots,a_n)}{2^{M+k+1}}+ \frac{1}{2})&=\frac{1}{2}-\frac{f_{\psi}(a_1,\dots,a_n)}{2^{M+k+1}}=\frac{1}{2}+\frac{f_{-\psi}(a_1,\dots,a_n)}{2^{M+k+1}} \\ &= r_{M,k}(f_{-\psi}(a_1,\dots,a_n)).
        \end{align*}
        Moreover, since by (the induction) hypothesis  
        $$f_{\tau(\psi)}(r_{M,k}(a_1),\dots,r_{M,k}(a_n)) \in [\frac{1}{2}-\frac{1}{2^{k-(l-1)+1}},\frac{1}{2}+\frac{1}{2^{k-(l-1)+1}}],$$ we obtain
        \begin{align*}
        f_{\tau(-\psi)}(r_{M,k}(a_1),\dots,r_{M,k}(a_n))&=1-f_{\tau(\psi)}(r_{M,k}(a_1),\dots,r_{M,k}(a_n)) \\ \in[\frac{1}{2}-\frac{1}{2^{k-l+2}},\frac{1}{2}+\frac{1}{2^{k-l+2}}] &\subseteq [\frac{1} {2}-\frac{1}{2^{k-l+1}},\frac{1}{2}+\frac{1}{2^{k-l+1}}].
        \end{align*}
        \item Let $\phi=\psi \land \gamma$.
        We assume
        \begin{align*}
        f_{\tau(\psi)}(r_{M,k}(a_1),
        \dots,r_{M,k}(a_n))&=r_{M,k} (f_{\psi}(a_1,\dots,a_n)),\\
        f_{\tau(\gamma)}(r_{M,k}(a_1),
        \dots,r_{M,k}(a_n))&=r_{M,k} (f_{\gamma}(a_1,\dots,a_n)).
        \end{align*}
        We obtain 
        \begin{align*}
            f_{\tau(\psi \land \gamma)}(r_{M,k}(a_1),\dots,r_{M,k}(a_n))&=f_{\tau(\psi) \land \tau(\gamma)}(r_{M,k}(a_1),\dots,r_{M,k}(a_n))\\=
            (f_{\tau(\psi)}\land f_{\tau(\gamma)})(r_{M,k}(a_1),\dots,r_{M,k}(a_n))&=((r_{M,k} \circ f_{\psi}) \land (r_{M,k} \circ f_{\gamma}))(a_1,\dots,a_n).
        \end{align*}
        Since $r_{M,k}$ is order preserving we obtain
        \begin{align*}
            ((r_{M,k} \circ f_{\psi}) \land (r_{M,k} \circ f_{\gamma}))(a_1,\dots,a_n)&=r_{M,k}((f_{\psi} \land f_{\gamma})(a_1,\dots,a_n))\\ &=
            r_{M,k}((f_{\psi \land \gamma})(a_1,\dots,a_n)).
        \end{align*}
        Since $$f_{\tau(\psi \land \gamma)}(r_{M,k}(a_1),\dots,r_{M,k}(a_n))= f_{\tau(\psi)}(r_{M,k}(a_1),\dots,r_{M,k}(a_n)),$$
        or $$f_{\tau(\psi \land \gamma)}(r_{M,k}(a_1),\dots,r_{M,k}(a_n))= f_{\tau(\gamma)}(r_{M,k}(a_1),\dots,r_{M,k}(a_n)),$$
        then necessarily  $$f_{\tau(\psi \land \gamma)}(r_{M,k}(a_1),\dots,r_{M,k}(a_n)) \in [\frac{1}{2}-\frac{1}{2^{k-(l-1)+1}},\frac{1}{2}+\frac{1}{2^{k-(l-1)+1}}]$$
        and therefore
        $$f_{\tau(\psi \land \gamma)}(r_{M,k}(a_1),\dots,r_{M,k}(a_n)) \in  [\frac{1}{2}-\frac{1}{2^{k-l+1}},\frac{1}{2}+\frac{1}{2^{k-l+1}}].$$
        
        \item The case $\phi=\psi \lor \gamma$ is analogous.
        \item     Now let $\phi$ be $\psi + \gamma$. We assume 
        \begin{align*}
            f_{\tau(\psi)}(r_{M,k}(a_1),
        \dots,r_{M,k}(a_n))&=r_{M,k} (f_{\psi}(a_1,\dots,a_n)),\\
        f_{\tau(\gamma)}(r_{M,k}(a_1),
        \dots,r_{M,k}(a_n))&=r_{M,k} (f_{\gamma}(a_1,\dots,a_n))
        \end{align*}
        and since $l=\depth{\phi}=\max(\depth{\psi},\depth{\gamma})+1$ we obtain
        \[f_{\tau(\psi)}(r_{M,k}(a_1),\dots,r_{M,k}(a_n))\in [\frac{1}{2}-\frac{1}{2^{k-l+2}},\frac{1}{2}+\frac{1}{2^{k-l+2}}].\]
        and 
        $$f_{\tau(\gamma)}(r_{M,k}(a_1),\dots,r_{M,k}(a_n))\in [\frac{1}{2}-\frac{1}{2^{k-l+2}},\frac{1}{2}+\frac{1}{2^{k-l+2}}].$$
        We distinguish two cases:
        \begin{enumerate}
        \item In case  $$(f_{\tau(\psi)}+f_{\tau(\gamma)})(r_{M,k}(a_1),\dots,r_{M,k}(a_n)) \geq 1$$ 
        we have $$f_{\tau(\psi) \oplus \tau(\gamma)}(r_{M,k}(a_1),\dots,r_{M,k}(a_n))=1$$ and $$f_{\tau(\psi) \otimes \tau(\gamma)}(r_{M,k}(a_1),\dots,r_{M,k}(a_n))=(f_{\tau(\psi)}+f_{\tau(\gamma)})(r_{M,k}(a_1),\dots,r_{M,k}(a_n))-1.$$
        Thus, we have
    \begin{align*}
    &f_{\tau(\psi+\gamma)}(r_{M,k}(a_1),\dots,r_{M,k}(a_n))\\ =& f_{(\tau(\psi) \oplus \tau(\gamma)) \otimes (\frac{1}{2} \oplus (\tau(\psi) \otimes \tau(\gamma)))}(r_{M,k}(a_1),\dots,r_{M,k}(a_n))\\ =&
    f_{1 \otimes (\frac{1}{2} \oplus (\tau(\psi) \otimes \tau(\gamma)))}(r_{M,k}(a_1),\dots,r_{M,k}(a_n))\\ =&
    f_{\frac{1}{2} \oplus (\tau(\psi) \otimes \tau(\gamma))}(r_{M,k}(a_1),\dots,r_{M,k}(a_n)\\ =&
    \min(1,f_{\tau(\psi) \otimes \tau(\gamma)}(r_{M,k}(a_1),\dots,r_{M,k}(a_n))+\frac{1}{2})\\ =&
             \min(1,(f_{\tau(\psi)}+f_{\tau(\gamma)})(r_{M,k}(a_1),\dots,r_{M,k}(a_n))-1+\frac{1}{2})\\ =&
        \min(1,(f_{\tau(\psi)}+f_{\tau(\gamma)})(r_{M,k}(a_1),\dots,r_{M,k}(a_n))-\frac{1}{2})
        \\ =&
        \min(1,(r_{M,k} \circ f_{\psi}+ r_{M,k} \circ f_{\gamma})(a_1,\dots,a_n)-\frac{1}{2}).
\end{align*}
        Moreover, by the induction hypothesis $$(r_{M,k} \circ f_{\psi}+ r_{M,k} \circ f_{\gamma}) (a_1,\dots,a_n)-\frac{1}{2} \in [\frac{1}{2}-\frac{1}{2^{k-l+1}},\frac{1}{2}+\frac{1}{2^{k-l+1}}]$$ and thus, we can conclude
        $$\min(1,(r_{M,k} (f_{\psi}+f_{\gamma})(a_1,\dots,a_n))-\frac{1}{2})=
        (r_{M,k} \circ f_{\psi}+ r_{M,k} \circ f_{\gamma})(a_1,\dots,a_n)-\frac{1}{2}.$$
        Therefore,
        \begin{align*}
          f_{\tau(\psi+\gamma)}(r_{M,k}(a_1),\dots,r_{M,k}(a_n))&=(r_{M,k} \circ f_{\psi}+ r_{M,k} \circ f_{\gamma})(a_1,\dots,a_n) -\frac{1}{2}\\ =\frac{(f_{\psi}+f_{\gamma})(a_1,\dots,a_n)}{2^{M+k+1}}+\frac{1}{2}&=r_{M,k}(f_{\psi+\gamma}(a_1,\dots, a_n)).  
        \end{align*}
         This completes the proof of the first case.
        \item In case $$(f_{\tau(\psi)}+f_{\tau(\gamma)})(r_{M,k}(a_1),\dots,r_{M,k}(a_n)) \leq 1,$$
        we have $$f_{\tau(\psi) \otimes \tau(\gamma)}(r_{M,k}(a_1),\dots,r_{M,k}(a_n))=0$$ and $$f_{\tau(\psi) \oplus \tau(\gamma)}(r_{M,k}(a_1),\dots,r_{M,k}(a_n))=(f_{\tau(\psi)}+f_{\tau(\gamma)})(r_{M,k}(a_1),\dots,r_{M,k}(a_n)).$$
      Thus, we have   
        \begin{align*}
        & f_{\tau(\psi+\gamma)}(r_{M,k}(a_1),\dots,r_{M,k}(a_n))\\ =&
        f_{(\tau(\psi) \oplus \tau(\gamma)) \otimes (\frac{1}{2} \oplus (\tau(\psi) \otimes \tau(\gamma)))}(r_{M,k}(a_1),\dots,r_{M,k}(a_n))\\ =&
        f_{(\tau(\psi) \oplus \tau(\gamma)) \otimes (\frac{1}{2} \oplus 0)}(r_{M,k}(a_1),\dots,r_{M,k}(a_n)) \\ =&
        f_{(\tau(\psi) \oplus \tau(\gamma)) \otimes \frac{1}{2}}(r_{M,k}(a_1),\dots,r_{M,k}(a_n))\\ =&
            \max(0,(f_{\tau(\psi)}+f_{\tau(\gamma)})(r_{M,k}(a_1),\dots,r_{M,k}(a_n))+\frac{1}{2}-1)\\ =&
            \max(0,(f_{\tau(\psi)}+f_{\tau(\gamma)})(r_{M,k}(a_1),\dots,r_{M,k}(a_n))-\frac{1}{2})\\ =&
            \max(0,(r_{M,k} \circ f_{\psi}+ r_{M,k} \circ f_{\gamma})(a_1,\dots,a_n)-\frac{1}{2}). 
        \end{align*}
   Since by the induction hypothesis 

   \begin{equation*}
       (r_{M,k} \circ f_{\psi}+r_{M,k} \circ f_{\gamma})(a_1,\dots,a_n) -\frac{1}{2}\in [\frac{1}{2}-\frac{1}{2^{k-l+1}},\frac{1}{2}+\frac{1}{2^{k-l+1}}]
   \end{equation*}  
   we can conclude that
   \begin{align*}
       f_{\tau(\psi+\gamma)}(r_{M,k}(a_1),\dots,r_{M,k}(a_n))&=(r_{M,k} \circ f_{\psi}+ r_{M,k} \circ f_{\gamma})(a_1,\dots,a_n))-\frac{1}{2}\\ =\frac{(f_{\psi}+f_{\gamma})(a_1,\dots,a_n)}{2^{M+k+1}}+\frac{1}{2}&=r_{M,k}(f_{\psi+\gamma}(a_1,\dots, a_n)).
   \end{align*}
   This completes the second case and thus concludes the proof. \qedhere
   \end{enumerate}
    \end{itemize}
\end{proof}
We will extend $\tau$ to open $\langAb$-formulas. 
For $\Phi$ of $\langAb$
define $\etau(\Phi)$ by replacing each occurrence of an atom $\phi=\psi$ with the atom $\tau(\phi)=\tau(\psi)$. 
By definition then, $\etau$ does not interfere with the boolean structure of $\Phi$.

\begin{lemma} \label{l:etau}
Let $M$ and $k$ be fixed nonnegative integers, let $\Phi$ be an open formula in $\langAb$ with $n$ variables
such that $k$ is an upper bound on the depth of terms in $\Phi$. 
Then for any $a_1,\dots, a_n \in [-2^{M},2^{M}]$ 
\begin{equation}
\R \vDash \Phi(a_1,\dots,a_n)\mbox{\  if and only if \ } \standardLhalf \vDash\etau(\Phi)(r_{M,k}(a_1),\dots,r_{M,k}(a_n)).
\end{equation}
Equivalently, for any $b_1,\dots, b_n \in [\frac{1}{2}-\frac{1}{2^{k+1}},\frac{1}{2}+\frac{1}{2^{k+1}}]$ 
\begin{equation}
\standardLhalf \vDash\etau(\Phi)(b_1,\dots,b_n)\mbox{\  if and only if \ } \R \vDash \Phi({r_{M,k}}^{-1}(b_1),\dots,{r_{M,k}}^{-1}(b_n)).
\end{equation}
\end{lemma}

\begin{proof}
We only prove the first equivalence, the second following from the first one using the fact that $r_{M,k}$ is a bijection  between $[-2^M,2^M]$ and $[\frac{1}{2}-\frac{1}{2^{k+1}},\frac{1}{2}+\frac{1}{2^{k+1}}]$.
The proof is by induction on boolean structure of $\Phi$.
    For an atom $\phi=\psi$ (where $\phi$, $\psi$ are terms), if
$f_{\phi}(a_1,\dots, a_n)=f_{\psi}(a_1,\dots, a_n)$ then $r_{M,k}(f_{\phi}(a_1,\dots,a_n))=r_{M,k}(f_{\psi}(a_1,\dots,a_n))$ and using 
Lemma~\ref{lm_r_m_k}, $$f_{\tau(\phi)}(r_{M,k}(a_1),\dots,r_{M,k}(a_n))=f_{\tau(\psi)}(r_{M,k}(a_1),\dots,r_{M,k}(a_n)),$$
which implies $\standardLhalf \vDash \etau(\phi=\psi)(r_{M,k}(a_1),\dots,r_{M,k}(a_n)).$
The other implication is obtained analogously using the fact that $r_{M,k}$ is one-to-one.

Now  assume the induction hypothesis for open formulas $\Psi$ and 
$\Delta$, that is  
\begin{align}
\R \vDash \Psi(a_1,\dots,a_n) &\text{ iff }\standardLhalf \vDash\etau(\Psi)(r_{M,k}(a_1),\dots,r_{M,k}(a_n))\mbox { and }\\
\R \vDash \Delta(a_1,\dots,a_n) &\text{ iff }\standardLhalf \vDash\etau(\Delta)(r_{M,k}(a_1),\dots,r_{M,k}(a_n)). 
\end{align}
First, since $$\R \nvDash \Psi(a_1,\dots,a_n) \text{ iff }\standardLhalf \nvDash\etau(\Psi)(r_{M,k}(a_1),\dots,r_{M,k}(a_n)),$$ and 
$\etau(\boneg \Psi)=\boneg \etau(\Psi)$, it follows that $$\R \vDash \boneg \Psi(a_1,\dots,a_n) \text{ iff } \standardLhalf \vDash\etau( \boneg\Psi)(r_{M,k}(a_1),\dots,r_{M,k}(a_n)).$$

Also,  $\etau(\Psi \boland \Delta)=\etau(\Psi) \boland \etau(\Delta)$ and thus clearly
$$\R \vDash (\Psi \boland \Delta)(a_1,\dots,a_n) \text{ iff } \standardLhalf \vDash \etau(\Psi \boland \Delta)(r_{M,k}(a_1),\dots,r_{M,k}(a_n)).$$
The induction step for the connective $\boor$ can be checked analogously.
\end{proof}

Let us recap what has been established. Let $M$ and $k$ be fixed nonnegative integers. 
Let $\Phi$ be an open formula with $n$ variables in $\langAb$ and term depth bounded by  $k$ be given.
Using the bijection $r_{M,k}$ of $[-2^M,2^M]$ 
onto $[\frac{1}{2}-\frac{1}{2^{k+1}},\frac{1}{2}+\frac{1}{2^{k+1}}]$  one can
map satisfying assignments of (variables in) $\Phi$ in $\R_{-1}$ to those of $\etau(\Phi)$ in $\standardL$ (which fall into $[\frac{1}{2}-\frac{1}{2^{k+1}},\frac{1}{2}+\frac{1}{2^{k+1}}]$), and 
conversely, any satisfying assignment to $\etau(\Phi)$ in the interval $[\frac{1}{2}-\frac{1}{2^{k+1}},\frac{1}{2}+\frac{1}{2^{k+1}}]$  of $\standardL$ can be mapped to a satisfying assignment to $\Phi$ (in $[-2^M,2^M]$).

The translation $\etau$, as presented, can lead to an exponential increase in term size.
For example  $\tau(\dots(x_1 + x_2)+x_3)+\dots+x_n)$ has   $2^{n-1}$ occurrences of variable $x_1$.
To attain a polynomial-size translation of open $\langpAb$-formulas to 
open $\langMV$-formulas,   complex terms need to be removed from $\Phi$ while preserving satisfying assignments,  before applying $\etau$. This can be done relying (again) on Tseitin transformations.
\begin{lemma}
\label{tseitin_etau_polysize}
Let $\Phi$ be a $\langAb$ formula. Let $\Phi'$ be the result of a Tseitin transformation of $\Phi$. 
Then for any assignment $v$ in $\standardLhalf$ one has that $v$ satisfies 
$\etau(\Phi)$ if and only if $\etau(\Phi')$ is satisfied by the assignment $v'$, where $v'$ uniquely extends  $v$ by sending each added variable $x_{\phi}$ to $v(\tau(\phi))$. Moreover, $\sharp(\etau(\Phi'))$ is polynomial in $\sharp(\Phi)$.
\end{lemma}

\begin{lemma} \label{l:bonus_variables}
Let $M,k$ be fixed nonnegative integers.
Consider the following set of equations in \langMV:
    \begin{enumerate}
        \item $z_1=\neg z_1$.
        \item $(z_{i+1} \oplus z_{i+1}=z_{i})$ for each $i \leq M+k$.
        \item $q =\neg (z_1 \oplus z_{M+k+1})$ and $r=\neg (z_1 \oplus z_{k+1})$.
    \end{enumerate}
Any assignment $v$ in $\standardL$ that satisfies all these equations has $v(z_i)=\frac{1}{2^i}$, $v(q)=\frac{1}{2}-\frac{1}{2^{M+k+1}}$ and $v(r)=\frac{1}{2}-\frac{1}{2^{k+1}}$.  
\end{lemma}

Notice that in the following lemma, the translation $S$ operates on $\langpAb$-formulas
and brings them to  $\langMV$-formulas.

\begin{lemma}
    Let $\Phi(x_1,\dots,x_n)$ be an open formula in the language $\langpAb$.
    Let $\Phi'$ be a Tseitin variant of $\Phi$.
     Assume that the variables $\{z_1,\dots,z_{M+k+1},r,q\}$ do not occur in  $\mathrm{Var}(\Phi')$.
    Let $\zeta (\Phi')$ be obtained from $\Phi'$ by replacing
    each occurrence of the constant $-1$ in $\Phi'$ with the variable $q$, applying $\etau$ to the resulting formula and then replacing all occurrences of the constant $\nicefrac{1}{2}$ with  the variable $z_1$.
    Let $k$ be the maximal depth of terms in $\Phi$ and, denoting $N\coloneqq \sharp(\Phi)$,
    set $M\coloneqq\lceil cN \log_2 N \rceil$ for some large integer $c$.
    Let $S(\Phi,M,k)$ be the following  open formula in $\langMV$: 
\begin{multline*}
    \zeta (\Phi') \boland     (z_1=\neg z_1) \boland (r=\neg (z_1 \oplus z_{k+1})) \boland\\
     (q =\neg (z_1 \oplus z_{M+k+1}) )\boland  \foo_{i \leq M+k} (z_{i+1} \oplus z_{i+1}=z_{i}) \boland \foo_{i \leq n} (r \leq x_i) \boland \foo_{i \leq n}(x_i \leq \neg r).
\end{multline*}
 Then $\Phi$ has a satisfying assignment in
 $\R_{-1}$ if and only if $S(\Phi,M,k)$ has a satisfying assignment in
$\standardL$.
 Moreover, $\sharp(S(\Phi,M,k))$ is polynomial in $\sharp(\Phi)$.
\end{lemma}
\begin{proof}
    Assume first that $\Phi(x_1,\dots, x_n)$ is satisfied in $\R_{-1}$. By Lemma~\ref{t:small models},
    there is a satisfying assignment $v(x_i)=a_i$ with $a_i \in   [-2^M,2^M]$. Then $\Phi'$ in variables $x_1,\dots, x_m$ (with $m\geq n$) is satisfied by
    any assignment that extends $v$ on $x_1,\dots, x_n$ by sending the extra variables $x_{n+1},\dots,x_m$ to the values $a_{n+1},\dots, a_m$ 
    of subterms in $\Phi$ in  the interval  $ [-2^{M+k+1},2^{M+k+1}]$.
    W.l.o.g.~we assume that $v$ itself is this extended, satisfying assignment.
    Apply Lemma~\ref{l:etau}, sending $-1$ to $\frac{1}{2}-\frac{1}{2^{M+k+1}}$ via the mapping $r_{M,k}$.
    Then $\zeta(\Phi' )$ has a satisfying assignment 
    $r_{M,k}(a_1),\dots, r_{M,k}(a_m)$     in $\standardL$
    and moreover, by construction, one that sends $q$ to $\frac{1}{2}-\frac{1}{2^{M+k+1}}$ and $z_1$ to $\frac{1}{2}$.
    By Lemma~\ref{l:bonus_variables}, the formula 
    \begin{equation} \label{eq:bonus variables}
        (z_1=\neg z_1) \boland (q =\neg (z_1 \oplus z_{M+k+1}) )\boland (r=\neg (z_1 \oplus z_{k+1}))\boland  \foo_{i \leq M+k} (z_{i+1} \oplus z_{i+1}=z_{i}) 
    \end{equation}
    has the unique satisfying assignment in $\standardL$, which
    sends $z_i$ to $\nicefrac{1}{2^i}$ for each $i \leq M+k$, $q$ to  $\frac{1}{2}-\frac{1}{2^{M+k+1}}$ and $r$ to $\frac{1}{2}-\frac{1}{2^{k+1}}$. 
    By construction, we have that $r_{M,k}(a_i)$ for $i\leq n$ belongs to   $[\frac{1}{2}-\frac{1}{2^{k+1}},\frac{1}{2}+\frac{1}{2^{k+1}}]$, so $ \foo_{i \leq n} (r \leq x_i) \boland \foo_{i \leq n}(x_i \leq \neg r)$ 
    is satisfied too.     

For the converse implication assume $S(\Phi,M,k)$ is satisfied in $\standardL$ by an assignment $v$. Let $v(x_i)=b_i$ for $i\leq n$.
In particular, then, Equation \ref{eq:bonus variables}
is satisfied and hence by Lemma~\ref{l:bonus_variables}, $v(z_1)=\frac{1}{2}$, $v(q)=  \frac{1}{2}-\frac{1}{2^{M+k+1}}$ and $v(r)=\frac{1}{2}-\frac{1}{2^{k+1}}$. This implies that $v(x_i)$ for $i\leq n$ are in $[\frac{1}{2}-\frac{1}{2^{k+1}},\frac{1}{2}+\frac{1}{2^{k+1}}]$. 
Using Lemma~\ref{tseitin_etau_polysize} one can easily show that any evaluation satisfying $\zeta(\Phi')$ will also satisfy $\zeta(\Phi)$. Thus $v$ satisfies $\zeta(\Phi)$ with $v(q)=\frac{1}{2}-\frac{1}{2^{M+k+1}}$ and $v(z_1)=\frac{1}{2}$. By Lemma~\ref{l:etau}
the assignment $v'(x_i)=r^{-1}_{M,k}(b_i)$ for $i\leq n$ satisfies
$\Phi$, which completes the proof.
\end{proof}

\begin{theorem}
The existential theory of $\R_{-1}$ reduces in polynomial time to the existential theory of $\standardL$. Therefore, the former theory is in NP.
\end{theorem}

It follows that the universal theory of $\R_{-1}$ is in coNP. 
It was already established in the paper \cite{Cintula-Jankovec-Noguera:SuperabelianLogics} that the equational theory of $\R_{-1}$ is coNP-hard. Hence the universal theory and the quasiequational theory of $\R_{-1}$  are coNP-complete and the existential theory is NP-complete.

\section{Reducing the \texorpdfstring{$\exists$}{∃}-theory of \texorpdfstring{$\standardL$}{L} to itself}
\label{s:embeddings}

This section investigates the composition of the reduction from $\standardL$ to $\R_{-1}$ and that from $\R_{-1}$ to $\standardLhalf$. 
First let us introduce a variant of the translation used in \cite{Cintula-Jankovec-Noguera:SuperabelianLogics}
which provided a faithful interpretation of   the universal theory of $\standardL$ (in the language $\langMV)$ in
the universal theory of  $\R_{-1}$ (in the language $\langpAb$), mentioned in Section \ref{s:intro}. 
Such a  mapping $\delta$ from $\formMV$ to $\formpAb$ is defined recursively as follows: 
\begin{multicols}{2}
\begin{enumerate}
\item $\delta(x) = (x \lor -1) \land 0$ for $x \in \mathrm{Var}$; 
\item $\delta(\phi \rightarrow \xi) = (\delta(\xi) - \delta(\phi)) \land 0$;
\item $\delta(\phi \otimes \psi)=(\delta(\phi) + \delta(\psi)) \lor -1$;
\item $\delta(0) = -1$;
\item $\delta(1) = 0$;
\item $\delta(\phi \circ \xi) = \delta(\phi)  \circ \delta(\xi)$ for $\circ \in \{\land,\lor\}.$
\end{enumerate}
\end{multicols}

\begin{lemma} \label{l:delta}
    Let $\phi(x_1,\dots, x_n)$ be a term in 
$\langMV$ with $n$ variables. Assume  $a_1,\dots, a_n$ are elements of the domain of $\standardL$.
Then $f_{\delta(\phi)}((a_1-1),\dots,(a_n-1))+1=f_{\phi}(a_1,\dots,a_n)$.

\end{lemma}
\begin{proof}
    As in the proof of \cite[Theorem 5.8]{Cintula-Jankovec-Noguera:SuperabelianLogics}.
\end{proof}

As in the case of  $\tau$ defined in Section \ref{s:reduction}, the mapping $\delta$ extends to a mapping on open formulas of $\langMV$: for $\Phi$ of $\langMV$ define $\edelta(\Phi)$ by replacing each occurrence of an atom $\phi=\psi$ with the atom $\delta(\phi)=\delta(\psi)$.

\begin{lemma} \label{l:edelta}
    Let $\Phi$ be an open formula in $\langMV$ with $n$ variables and let $a_1,\dots,a_n \in [0,1]$.
    Then
    $$\standardL \vDash \Phi(a_1,\dots,a_n) \text{ if and only if } \R_{-1} \vDash \edelta(\Phi) (a_1-1,\dots,a_n-1).$$
\end{lemma}
\begin{proof}
    As in the proof of Lemma~\ref{l:etau}.
\end{proof}

Our aim is to compose $\delta$ with $\tau$. This is not immediately possible, since $\tau$ is defined only on $\formAb$ (lacking the constant $-1)$. First, therefore, the definition of $\tau$ will be extended to its variant $\tau_q$.
Let $\phi \in \formpAb$ and let $q$ be a variable not used in $\phi$. The function  $\tau_q(\phi)$ is defined recursively as follows:
\begin{enumerate}
    \item $\tau_q(x)=(x \lor q) \land \frac{1}{2} $ for $x \in {\mathrm Var}(\phi)$;
    \item $\tau_q(0)=\frac{1}{2}$;
    \item $\tau_q(-1)=q$;
    \item $\tau_q(-\phi)=\neg(\tau_q(\phi))$;
    \item $\tau_q(\phi \circ \psi)=\tau_q(\phi) \circ \tau_q(\psi)$ for $\circ \in \{\lor,\land\}$;
    \item $\tau_q(\phi+\psi)=(\tau_q(\phi) \oplus \tau_q(\psi)) \otimes (\frac{1}{2} \oplus (\tau_q(\phi) \otimes \tau_q(\psi)))$.
\end{enumerate}

We now look into the composition $\tau_q \circ \delta$ which maps $\formMV$ into $\formMVhalf$. The following lemma shows that, as far as the semantics provided by $\standardL$ is concerned, the composition of these two translations can be expressed by a function that increases the length of a term polynomially, rather than  exponentially as one might expect, given the behaviour of $\tau$ discussed in Section \ref{s:reduction}.

\begin{lemma} \label{l:sigma}
    Let $\phi \in \formMV$ and let $q$ be a variable not among $\mathrm{Var}(\phi)$.
    The functions $f_{\tau_q(\delta(\phi))}$ and $f_{\sigma_q(\phi)}$ coincide on the domain of $\standardL$, where the mapping $\sigma_q$ on $\formMV$ is defined recursively as follows:
\setlength{\columnsep}{0pt}
\begin{multicols}{2}
  \begin{enumerate}
    \item $\sigma_q(x)=(x \lor q) \land \frac{1}{2}$ for $x \in {\mathrm Var}(\phi)$;
    \item $\sigma_q(1)=\frac{1}{2}$;
    \item $\sigma_q(0)=q$;
    \item $\sigma_q(a \rightarrow b)=((\sigma_q(a) \rightarrow \sigma_q(b)) \otimes \frac{1}{2}) \land \frac{1}{2}$;
    \item $\sigma_q(x \otimes y)=((\sigma_q(x) \oplus \sigma_q(y)) \otimes \frac{1}{2}) \lor q$;
    \item $\sigma_q(\phi \circ \xi) = \sigma_q(\phi)  \circ \sigma_q(\xi)$ for $\circ \in \{\land,\lor\}$.
\end{enumerate}
\end{multicols}
Moreover, for any term $\psi$ in $\langMV$, $\sharp(\sigma_q(\psi))$ is polynomial in $\sharp(\psi)$.
\end{lemma}
\begin{proof}
    By induction on term  depth.
    \begin{itemize}
    \item Base case (depth 0): By definition $\sigma_q(x)=(x \lor q) \land \frac{1}{2}=\tau_q(\delta(x))$, $\sigma_q(1)=\frac{1}{2}=\tau_q(\delta(1))$ and $\sigma_q(0)=q=\tau_q(\delta(0))$.
    \end{itemize}
    Now assume that $f_{\sigma_q(\psi)}=f_{\tau_q(\delta(\psi))}$  and  $f_{\sigma_q(\gamma)}=f_{\tau_q(\delta(\gamma))}$.
    \begin{itemize}
\item Clearly, for $\phi=\psi \circ \gamma$ we have $f_{\tau_q(\delta(\psi \circ \gamma))}=f_{\sigma_q(\psi \circ\gamma)}$ for $\circ \in \{\lor,\land\}$.
  \item For $\phi=\psi \rightarrow \gamma$ we have $$f_{\tau_q(\delta(\psi \rightarrow \gamma))}=f_{\tau_q((\delta(\gamma)-\delta(\psi)) \land 0)}=f_{(( \neg \tau_q(\delta(\psi)) \oplus \tau_q(\delta(\gamma))) \otimes (\frac{1}{2} \oplus  (\neg \tau_q(\delta(\psi)) \otimes \tau_q(\delta(\gamma))))) \land \frac{1}{2}}.$$
   By the induction hypothesis this is equivalent to
   \begin{align*}
       &f_{((\neg \sigma_q(\psi) \oplus \sigma_q(\gamma)) \otimes (\frac{1}{2} \oplus (\neg \sigma_q(\psi) \otimes \sigma_q(\gamma)))) \land \frac{1}{2}}\\=& \max(0, f_{\neg \sigma_q(\psi) \oplus \sigma_q(\gamma)}+f_{\frac{1}{2} \oplus (\neg \sigma_q(\psi) \otimes \sigma_q(\gamma))}-1) \land \frac{1}{2} .
   \end{align*}
   Let us distinguish two cases. 
   \begin{enumerate}
       \item 
   First assume $(f_{\neg\sigma_q(\psi)}+f_{\sigma_q(\gamma)})(a_1,\dots,a_n) \geq 1$.
   Thus,
   $$f_{\neg \sigma_q(\psi) \oplus \sigma_q(\gamma)}(a_1,\dots,a_n)=1$$ and consequently
\begin{align*}
    &\max(0, (f_{\neg \sigma_q(\psi) \oplus \sigma_q(\gamma)} +f_{\frac{1}{2} \oplus (\neg \sigma_q(\psi) \otimes \sigma_q(\gamma))})(a_1,\dots,a_n)-1) \land \frac{1}{2}\\ =&\max(0,f_{\frac{1}{2} \oplus (\neg \sigma_q(\psi) \otimes \sigma_q(\gamma))}(a_1,\dots,a_n)) \land \frac{1}{2}=
    \frac{1}{2}\\ =&f_{((\sigma_q(\psi) \rightarrow \sigma_q(\gamma)) \otimes \frac{1}{2}) \land \frac{1}{2}}(a_1,\dots,a_n)=f_{\sigma_q(\psi \rightarrow \gamma)}(a_1,\dots,a_n).
\end{align*}
\item For the other case assume $(f_{\neg\sigma_q(\psi)}+f_{\sigma_q(\gamma)})(a_1,\dots,a_n) \leq 1$. Then we have 
$$f_{((\neg \sigma_q(\psi) \oplus \sigma_q(\gamma)) \otimes (\frac{1}{2} \oplus (\neg \sigma_q(\psi) \otimes \sigma_q(\gamma)))) \land \frac{1}{2}}=f_{((\neg \sigma_q(\psi) \oplus \sigma_q(\gamma)) \otimes \frac{1}{2}) \land \frac{1}{2}}=f_{\sigma_q(\psi \rightarrow \gamma)}.$$ This completes the second case.
\end{enumerate}
\item For $\phi=\psi \otimes \gamma$ we have $$f_{\tau_q(\delta(\psi \otimes \gamma))}=f_{\tau_q((\delta(\psi) + \delta(\gamma)) \lor -1)}=f_{((\tau_q(\delta(\psi)) \oplus \tau_q(\delta(\gamma))) \otimes (\frac{1}{2} \oplus (\tau_q(\delta(\psi)) \otimes \tau_q(\delta(\gamma))))) \lor q}.$$
By induction hypothesis this is equal to 
$$f_{((\sigma_q(\psi) \oplus \sigma_q(\gamma)) \otimes (\frac{1}{2} \oplus (\sigma_q(\psi) \otimes \sigma_q(\gamma)))) \lor q}.$$
Since $f_{\sigma_q(\gamma)}(a_1,\dots,a_n)\leq \frac{1}{2}$ for each $a_1,\dots,a_n \in \standardL$ and for each $\gamma \in \formMV$, we have for each $a_1,\dots,a_n \in \standardL$ that $f_{\sigma_q(\psi) \otimes \sigma_q(\gamma)}(a_1,\dots,a_n)=0.$
Thus, $$f_{\tau_q(\delta(\psi \otimes \gamma))}=f_{((\sigma_q(\psi) \oplus \sigma_q(\gamma)) \otimes \frac{1}{2} ) \lor q}.$$
This completes the second case and also the  proof. \qedhere
    \end{itemize} 
    \end{proof}
We extend definitions of $\tau_q$ and $\sigma_q$ to open formulas.
For $\Phi$ of $\langpAb$ define $\etau_q(\Phi)$ by replacing each occurrence of an atom $\phi=\psi$ with the atom $\tau_q(\phi)=\tau_q(\psi).$ 
Similarly for $\Phi$ of $\langMV$ define  $\esigma_q(\Phi)$ by replacing each occurrence of an atom $\phi=\psi$ with the atom $\sigma_q(\phi)=\sigma_q(\psi)$.

\begin{corollary} \label{c:esigma}
    Let $\Psi$ be open formula of $\langMV$ with $n$ variables and let $q\not\in\mathrm{Var}(\Psi)$. 
    Let $a_1,\dots a_n \in [0,1]$.
    Then $$\standardLhalf \vDash\etau_q(\edelta(\Psi))(a_1,\dots,a_n) \text{ if and only if } \standardLhalf \vDash \esigma_q(\Psi)(a_1,\dots,a_n).$$
\end{corollary}

Notice that if $\phi$ is a $\langMV$-term and $v$ an assignment in $\R_{-1}$ such that $v(x_i) \in [-1,0]$ for each $x_i \in \mathrm{Var}(\phi)$, then for each subterm $\psi$ of $\phi$,  $v(\delta(\psi)) \in [-1,0]$. 
Using this observation, we obtain the following variant of Lemma~\ref{l:etau}, now featuring the translation $\etau_q$ in the role of $\etau$, and moreover avoiding the assumption regarding the depth of terms in the open first-order formula. 
The proof of this lemma is a variant of those of Lemma~\ref{l:tau} and Lemma~\ref{l:etau} with only minor modifications. 

\begin{lemma} \label{l:etau-modified}
Let $M$ and $k$ be fixed nonnegative integers,  $\Phi(x_1, \dots, x_n)$  an open formula in the language $\langMV$ and  $q\not\in\mathrm{Var}(\Phi)$.
For any assignment $v$ in $\R_{-1}$ where $v(x_i) = a_i \in [-1,0]$, let $v'$ be an assignment in $\standardLhalf$ with
\setlength{\abovedisplayskip}{3pt}
\setlength{\abovedisplayshortskip}{3pt}
\begin{equation*} \label{eq:evaluation_v}
    v'(x_i) = r_{M,k}(a_i) \text{ for each } i \leq n  \text{ and } v'(q) = \frac{1}{2}-\frac{1}{2^{M+k+1}}.
\end{equation*}
Then  $v$  satisfies the formula $\edelta(\Phi)$ in the algebra $\R_{-1}$ if and only if $v'$ satisfies the formula $\etau_q(\edelta(\Phi))$ in $\standardLhalf$. 

This is equivalent to the following. For any assignment $v$ in $\standardLhalf$ where $v(x_i) = b_i \in [\frac{1}{2}-\frac{1}{2^{M+k+1}},\frac{1}{2}]$
let $v'$ be an assignment in $\R_{-1}$ with $v'(x_i)=r_{M,k}^{-1}(b_i)$. Assume that $v(q) = \frac{1}{2}-\frac{1}{2^{M+k+1}}$. Then  $v$  satisfies the formula $\etau_q(\edelta(\Phi))$ in the algebra $\standardLhalf$ if and only if $v'$ satisfies the formula $\edelta(\Phi)$ in $\R_{-1}$. 
\end{lemma}

\begin{theorem}[A translation of the existential theory of $\standardL$ to itself]
    Let $\Phi(x_1,\dots,x_n)$ be an open formula of $\langMV$.
   Assume that the variables $\{z_1,\dots,z_{M+k+1},q\}$ do not occur in $\Phi$. 
    Let $M,k \in \N$.
    Let $\zeta (\Phi)$ be obtained from $\Phi$ by applying $\esigma_q$ on $\Phi$ and then replacing all occurrences of the constant $\frac{1}{2}$ with the variable $z_1$.
    Let $S(\Phi,M,k)$ be the following open formula in $\langMV$:
\begin{equation*}
    \zeta (\Phi) \boland     (z_1=\neg z_1)
    \boland (q =\neg (z_1 \oplus z_{M+k+1}) )\boland   \foo_{i \leq M+k} (z_{i+1} \oplus z_{i+1}=z_{i}). 
\end{equation*}
 Then $\Phi$ has a satisfying assignment in $\standardL$
if and only if $S(\Phi,M,k)$ has a satisfying assignment in $\standardL$.
 Moreover, $\sharp(S(\Phi,M,k))$ is polynomial in $\sharp(\Phi)$.   
\end{theorem}
\begin{proof}
Assume first that $\Phi(x_1,\dots, x_n)$ has a satisfying assignment $v(x_i)=a_i$ with $a_i \in \standardL$. By Lemma~\ref{l:edelta} $\edelta(\Phi)$ has a satisfying assignment $x_i \mapsto a_i-1$ with $(a_i-1) \in [-1,0]$.

By Lemma~\ref{l:etau-modified} $\etau_q(\edelta(\Phi))$ has a satisfying assignment $x_i \mapsto r_{M,k}(a_i-1)$ and $q \mapsto \frac{1}{2}-\frac{1}{2^{M+k+1}}$ in $\standardLhalf$. By Lemma~\ref{c:esigma} $\esigma_q(\Phi)$ is satisfied by the very same assignment. 
Consequently, $\zeta(\Phi)$ has satisfying assignment $x_i \mapsto r_{M,k}(a_i-1), q \mapsto \frac{1}{2}-\frac{1}{2^{M+k+1}}$ and $z_1 \mapsto \frac{1}{2}$ in $\standardL$.

By Lemma~\ref{l:bonus_variables}, the formula $(z_1=\neg z_1) \boland (q =\neg (z_1 \oplus z_{M+k+1}) )\boland  \foo_{i \leq M+k} (z_{i+1} \oplus z_{i+1}=z_{i}) $ has the unique satisfying assignment in $\standardL$, which
    sends $z_i$ to $\nicefrac{1}{2^i}$ for each $i \leq M+k$ and $q$ to   $\frac{1}{2}-\frac{1}{2^{M+k+1}}$.   

Conversely, assume $S(\Phi,M,k)$ has a satisfying assignment $v$ in $\standardL$. The conjunct 
$$(z_1=\neg z_1) \boland (q =\neg (z_1 \oplus z_{M+k+1}) )\boland  \foo_{i \leq M+k} (z_{i+1} \oplus z_{i+1}=z_{i}) $$
forces by Lemma~\ref{l:bonus_variables} that $v(z_1)=\frac{1}{2}$ and $v(q)=\frac{1}{2}-\frac{1}{2^{M+k+1}}$. Observe that by the definition of $\zeta$, every occurrence of a variable $x_i$ in $\zeta(\Phi)$ appears strictly within the term $(x_i \lor q) \land z_1$ for each $i \leq n$. Therefore, we can assume without loss of generality that $\frac{1}{2}-\frac{1}{2^{M+k+1}}=v(q)\leq v(x_i)\leq v(z_1)=\frac{1}{2}$ for all $i \leq n$.
Since $v$ satisfies $S(\Phi,M,k)$, it must satisfy $\zeta(\Phi)$. By definition of $\zeta$, this means the formula $\esigma_q(\Phi)$ is satisfied by the assignment $v$ in $\standardLhalf$. By Lemma~\ref{c:esigma} the assignment $v$ also satisfies $\etau_q(\edelta(\Phi))$ in $\standardLhalf$.
We can now apply Lemma~\ref{l:etau-modified}. 

Since $\frac{1}{2}-\frac{1}{2^{M+k+1}} \leq v(x_i) \leq \frac{1}{2}$ and $v(q)=\frac{1}{2}-\frac{1}{2^{M+k+1}}$, by Lemma~\ref{l:etau-modified} the assignment
$v'(x_i)= r_{M,k}^{-1}(v(x_i)) $ 
satisfies $\edelta(\Phi)$ in $\R_{-1}$.
Finally, applying Lemma~\ref{l:edelta}, the assignment $v''(x_i)=v'(x_i)+1$ satisfies the original formula $\Phi$ in $\standardL$. This completes the proof.
\end{proof}

The last theorem can be viewed as providing  a way of testing the validity of existential 
sentences in $\standardL$ on a subinterval of its domain (rather than the full domain), provided the value $\nicefrac{1}{2}$ belongs to this interval. The statement relies on the well-known fact that rational values in $[0,1]$ are implicitly definable by $\langMV$-terms, and moreover, the size of the defining term is polynomial in the length of the positional representation of the rational number in question.       
\medskip

\section*{Acknowledgements}
Both authors were supported by ERDF-Project \emph{Knowledge in the Age of Distrust}, No.\ CZ.02.01.01/00/23\_025/0008711.
The second author was also supported from the project SVV-2025-260837.

\bibliography{mfl-edited}

\end{document}